\documentclass{amsart}

\usepackage{amssymb,amsmath,color}
\usepackage{amsthm}

\usepackage{url}
\usepackage{tikz}
\usepackage[enableskew]{youngtab}
\usepackage{ytableau,varwidth}

\usepackage[colorlinks=true, allcolors=blue]{hyperref}

\makeatletter
\@namedef{subjclassname@2020}{%
  \textup{2020} Mathematics Subject Classification}
\makeatother

\allowdisplaybreaks

\newcommand{\leg}[2]{\genfrac{(}{)}{}{}{#1}{#2}}

\newcommand{\qbin}[2]{\genfrac{[}{]}{0pt}{}{#1}{#2}_q}
\renewcommand{\Re}{\operatorname{Re}}

\newtheorem{theorem}{Theorem}[section]
\newtheorem{corollary}[theorem]{Corollary}

\newtheorem{lemma}[theorem]{Lemma}

\theoremstyle{definition}

\begin{document}

\title{Congruences for sums\\ of MacMahon's $q$-Catalan polynomials}

\author[T. Amdeberhan]{Tewodros Amdeberhan} 

\address{Department of Mathematics,
Tulane University, New Orleans, LA 70118, USA}
\email{tamdeber@tulane.edu}

\author[R. Tauraso]{Roberto Tauraso}

\address{Dipartimento di Matematica, 
Università di Roma ``Tor Vergata'', 00133 Roma, Italy}
\email{tauraso@mat.uniroma2.it}


\begin{abstract} One variant of the $q$-Catalan polynomials is defined in terms of Gaussian polynomials by $\mathcal{C}_k(q)=\qbin{2k}{k}-q\qbin{2k}{k+1}$. Liu studied congruences of the form $\sum_{k=0}^{n-1} q^k\mathcal{C}_k$ modulo the cyclotomic polynomial $\Phi_n(q)^2$, provided that $n\equiv\pm 1\pmod3$. Apparently the case $n\equiv 0\pmod3$ has been missing from the literature. It is our primary purpose to fill this gap by the current work. In addition, we discuss certain fascinating link to Dirichlet character sum identities.
\end{abstract}

\maketitle

\section{Introduction}

There are several possible $q$-analogs of the Catalan numbers $\frac{1}{k+1}\binom{2k}{k}$. 
Here we consider the MacMahon's $q$-Catalan polynomials which are defined in terms of the $q$-binomial coefficients (Gaussian polynomials) as
$$
\mathcal{C}_k(q):=\frac{1-q}{1-q^{k+1}}\qbin{2k}{k}=\qbin{2k}{k}-q\qbin{2k}{k+1}.
$$
The first few of these $q$-Catalan are $\mathcal{C}_0(q)=\mathcal{C}_1(q)=1$, $\mathcal{C}_2(q)=1+q^2$ and $\mathcal{C}_3(q)=1+q^2+q^3+q^4+q^6$. Notice that $\mathcal{C}_k(q)$ is a polynomial in $q$ and it reduces to the ordinary Catalan number as $q\to 1$. Moreover, $\mathcal{C}_k(q)$ has a natural enumerative meaning. Indeed MacMahon~\cite[Vol. 2, p. 214]{MacMahon15} established that
$$\mathcal{C}_k(q)=\sum_{w}q^{\operatorname{maj}(w)}$$
where $w$ ranges over all ballot sequences $a_1a_2\cdots a_{2k}$ (i. e. any permutation of the multiset $\{0^k, 1^k\}$ such that in any subword $a_1a_2\cdots a_{i}$, there are at least as many $0$s as there are $1$s) and 
$$\operatorname{maj}(w):=\sum_{\{i: a_i>a_{i+1}\}}\!\!\!\!  i$$
is the \emph{major index} of $w$ (see also the survey \cite[Section 3]{FurlingerHofbauer85} and \cite[Problem A43]{Stanley15}).

\noindent In the present work, however, we focus on a problem of number-theoretic interest: congruences.
The second author in \cite[Theorem 6.1]{Tauraso12} achieved the following result 
\begin{align*}
\sum_{k=0}^{n-1}q^k \mathcal{C}_k(q)\equiv \begin{cases}
\displaystyle q^{\lfloor \frac{n}{3}\rfloor}&\text{if $n\equiv 0,1 \pmod{3}$}\\
\displaystyle -1-q^{\frac{2n-1}{3}}&\text{if $n\equiv 2 \pmod{3}$}
\end{cases}\pmod{\Phi_n(q)}
\end{align*}
where
\begin{align*}
\Phi_n(q)=\prod_{\substack{1\le k \le n\\
(n,k)=1}}(q-e^{\frac{2k\pi i}{n}})
\end{align*}
denotes the $n^{th}$-cyclotomic polynomial.

Afterwards, a stronger version modulo $\Phi_n(q)^2$, has been proved by J.-C. Liu in
\cite[Theorem 1]{Liu20},
$$
\sum_{k=0}^{n-1}q^k \mathcal{C}_k(q)\equiv \begin{cases}
q^{\frac{n^2-1}{3}}-\frac{n-1}{3}(q^n-1)\;&\text{if $n\equiv 1\pmod{3}$,}
\\[3pt]
-q^{\frac{n^2-1}{3}}-q^{\frac{n(2n-1)}{3}}\;&\text{if $n\equiv 2\pmod{3}$,}
\end{cases}\pmod{\Phi_n(q)^2}
$$
where the case $n\equiv 0\pmod{3}$ is \emph{not} covered. Our main aim at present is to fill this gap as stated next.

\begin{theorem}\label{mainC}
If $n$ is a positive integer divisible by $3$, then 
$$\sum_{k=0}^{n-1}q^k \mathcal{C}_k(q)\equiv
q^{\frac{n(2n+1)}{3}} +\frac{1}{3}(q^n-1)\Big(2+(n+1)q^{\frac{2n}{3}}\Big)
\pmod{\Phi_n(q)^2}.$$
\end{theorem} 

\noindent As we will explain in more detail below, the above theorem holds as soon as we prove the following more manageable identity, which is of interest in its own right.

\begin{theorem}\label{mainT}
If $n$ is a positive integer divisible by $3$ and $q$ is a primitive $n^{th}$-root of unity, then
\begin{equation}\label{main}
\sum_{k=1}^{\frac{n}{3}}\frac{(-1)^kq^{\frac{k(3k-1)}{2}}}{1-q^{3k-1}}
+\sum_{k=1}^{\frac{n}{3}-1}\frac{(-1)^kq^{\frac{k(3k+5)}{2}}}{1-q^{3k}}
=\frac{1}{6}\left(2+(n+1)q^{\frac{2n}{3}}\right).
\end{equation}
\end{theorem}

Notice that, according to \cite[Lemma 3]{Liu20}, our Theorem~\ref{mainT} mirrors that of
\begin{align*}
\sum_{k=1}^{\lfloor \frac{n}{3}\rfloor}\frac{(-1)^kq^{\frac{k(3k-1)}{2}}}{1-q^{3k-1}}
+\sum_{k=1}^{\lfloor \frac{n-1}{3}\rfloor}\frac{(-1)^kq^{\frac{k(3k+5)}{2}}}{1-q^{3k}}
\equiv
\begin{cases}
-\frac{n-1}{6}\quad &\text{if $n\equiv 1\pmod{3},$}\\
0\quad &\text{if $n\equiv 2\pmod{3}$}.
\end{cases}
\end{align*}

\smallskip
The remainder of the paper is organized as follows. In Section~\ref{reduce}, we present a reduction of our main result Theorem~\ref{mainC} into Theorem~\ref{mainT}. Section~\ref{prepare} contains preliminary results which we need towards the proof of Theorem~\ref{mainT}. Sections~\ref{evencase} and \ref{oddcase} split up Theorem~\ref{mainT} according to the parity of $n$ and we provide the corresponding proofs respectively therein. Finally, in Section~\ref{fini} we consider a conversion of one particular identity coming from \eqref{main}, into a trigonometric format and a remarkable implication of it in the language of character sums.

\section{Reducing Theorem \ref{mainC} to Theorem \ref{mainT}} \label{reduce}

We recall that the Gaussian $q$-binomial coefficients are defined as
\begin{align*}
\qbin{n}{k}=\begin{cases}
\displaystyle\frac{(q;q)_n}{(q;q)_k(q;q)_{n-k}} &\text{if $0\leq k\leq n$},\\[10pt]
0 &\text{otherwise,}
\end{cases}
\end{align*}
where the $q$-shifted factorial is given by $(a;q)_n=(1-a)(1-aq)\cdots(1-aq^{n-1})$ for $n\ge 1$ and $(a;q)_0=1$. 

By \cite[Theorem 1.2]{LiuPetrov20}, 
\begin{equation}\label{LP}
\sum_{k=0}^{n-1}q^k\qbin{2k}{k}\equiv 
\leg{n}{3}q^{\frac{n^2-1}{3}}\pmod{\Phi_n(q)^2}
\end{equation}
where $\leg{\,\cdot\,}{\cdot}$ denotes the Legendre symbol.
In the same vain, we also revive the identity \cite[Theorem 4.2]{Tauraso13},
$$
\sum_{k=0}^{n-1}q^{k+1}\qbin{2k}{k+1}=\sum_{k=1}^{n}\leg{k-1}{3}q^{\frac{1}{3}\left(2k^2-k
\left(\frac{k-1}{3}\right)\right)}\qbin{2n}{n+k}.
$$
Let $1\leq k\leq n-1$. Then, the $q$-analog \cite[Theorem 2.2]{Sagan92}
$$\qbin{an+b}{cn+d}\equiv \binom{a}{c} \qbin{b}{d} \pmod{\Phi_n(q)}$$
of Lucas' classical binomial congruence combined with $(1-q^n)\equiv 0 \pmod{\Phi_n(q)}$, and the fact that
$$\qbin{n-1}{k-1}=
\prod_{j=1}^{k-1}\frac{1-q^{n-j}}{1-q^j}
=q^{-\frac{k(k-1)}{2}}
\prod_{j=1}^{k-1}\frac{q^j-q^{n}}{1-q^j}
\equiv (-1)^{k-1} q^{-\frac{k(k-1)}{2}}
\pmod{\Phi_n(q)}$$
immediately imply that
\begin{align*}
\qbin{2n}{n+k}&=\frac{1-q^{2n}}{1-q^{n+k}}\qbin{2n-1}{n+k-1}
\equiv (1-q^{n})\cdot \frac{2}{1-q^{k}}\binom{1}{1}\qbin{n-1}{k-1}\\
&\equiv (q^{n}-1)\cdot \frac{2(-1)^k q^{-\frac{k(k-1)}{2}}}{1-q^{k}} 
\pmod{\Phi_n(q)^2}.
\end{align*}
Therefore, we find reliable verity to declare that 
\begin{align*}
\sum_{k=0}^{n-1}q^{k+1}\qbin{2k}{k+1}&\equiv
\sum_{k=0}^{n-1}\leg{k-1}{3}q^{\frac{1}{3}\left(2k^2-k
\leg{k-1}{3}\right)}\qbin{2n}{n+k}\\
&\equiv 2(q^n-1)\sum_{k=1}^{n-1}
\leg{k-1}{3}q^{\frac{1}{3}\left(2k^2-k
\leg{k-1}{3}\right)}\frac{(-1)^kq^{-\frac{k(k-1)}{2}}}{1-q^k}\\
&\equiv -2(q^n-1)\left(\sum_{k=1}^{\lfloor \frac{n}{3}\rfloor}\frac{(-1)^kq^{\frac{k(3k-1)}{2}}}{1-q^{3k-1}}
+\sum_{k=1}^{\lfloor \frac{n-1}{3}\rfloor}\frac{(-1)^kq^{\frac{k(3k+5)}{2}}}{1-q^{3k}}\right)
\end{align*}
holds modulo $\Phi_n(q)^2$.

Hence, by putting this congruence together with \eqref{LP} into the very definition of $\mathcal{C}_k(q)$, we easily obtain that, when $n$ is divisible by $3$, Theorem \ref{mainC} is indeed equivalent to Theorem \ref{mainT}.

\section{Preparing our proof of Theorem \ref{mainT}} \label{prepare}

Henceforth, we replace $n$ with $3n$ so that our target in \eqref{main} amounts to proving 
\begin{equation}\label{main3n}
\sum_{k=1}^{n}\frac{(-1)^kq^{\frac{k(3k-1)}{2}}}{1-q^{3k-1}}
+\sum_{k=1}^{n-1}\frac{(-1)^kq^{\frac{k(3k+5)}{2}}}{1-q^{3k}}
=\frac{1}{3}+\frac{3n+1}{6}\,q^{2n}.
\end{equation}
 
In order to establish this identity, we need the next two results.

\begin{lemma} We have that for any complex number $z$,
\begin{align}\label{mid}
\sum_{k=1}^n \frac{(-1)^k z^{\frac{k(3k-1)}2}}{1-z^{3k-1}}&+\sum_{k=1}^{n-1}\frac{(-1)^k z^{\frac{k(3k+5)}2}}{1- z^{3k}} 
\nonumber\\
&= \frac{(-1)^{n-1}}{2}\sum_{k=1}^{n-1}\frac{z^{\frac{k(3n+2)}{2}}}{1+z^{\frac{3k}{2}}} 
+\frac{1}{2}\sum_{k=1}^{n-1}\frac{(-1)^k z^{\frac{k(3n+2)}{2}}}{1- z^{\frac{3k}{2}}} \nonumber\\
&\quad + \, \sum_{k=1}^n\frac1{1-z^{3k-1}}-\sum_{k=1}^{\lfloor \frac{n+1}{2}\rfloor}\frac{1}{1-z^{3k-2}}
-\frac{2n-1+(-1)^n}{4}.
\end{align}
\end{lemma}
\begin{proof} Employing partial fractions and after further rearrangement, we obtain
\begin{align*}
\sum_{k=1}^n \frac{(-1)^k z^{\frac{k(3k-1)}2}}{1-z^{3k-1}}
&= \frac12\sum_{k=1}^n \frac{(-1)^k z^{\frac{k(3k-1)}2}}{1-z^{\frac{3k-1}2}}+ \frac12\sum_{k=1}^n \frac{(-1)^k z^{\frac{k(3k-1)}2}}{1+z^{\frac{3k-1}2}} \\
&=  \frac12\sum_{k=1}^n \frac{(-1)^k ((z^{\frac{3k-1}2})^k-1+1)}{1-z^{\frac{3k-1}2}}- \frac12\sum_{k=1}^n \frac{-(-z^{\frac{3k-1}2})^k+1-1}{1-(-z^{\frac{3k-1}2})} \\
& = -\frac12\sum_{k=1}^n \sum_{j=0}^{k-1}((-1)^k+(-1)^j) z^{\frac{j(3k-1)}2}  \\
& \qquad \qquad \qquad 
+ \frac12\sum_{k=1}^n \left(\frac1{1+z^{\frac{3k-1}2}}+\frac{(-1)^k}{1-z^{\frac{3k-1}2}}\right)\\
& = -\frac12\sum_{k=1}^n \sum_{j=0}^{k-1}((-1)^k+(-1)^j) z^{\frac{j(3k-1)}2}  \\
& \qquad \qquad \qquad 
+ \sum_{k=1}^n\frac1{1-z^{3k-1}}-\sum_{k=1}^{\lfloor \frac{n+1}{2}\rfloor}\frac{1}{1-z^{3k-2}}.
\end{align*}
Continuing with additional algebraic manipulation leads to
\begin{align*}
&\sum_{k=1}^n\sum_{j=0}^{k-1}((-1)^k+(-1)^j) z^{\frac{j(3k-1)}2}
= \sum_{j=0}^{n-1}\sum_{k=j+1}^n((-1)^k+(-1)^j) z^{\frac{j(3k-1)}2} \\
&\quad= \frac{2n-1+(-1)^n}2+ 2\sum_{j=1}^{n-1}\frac{(-1)^j z^{\frac{j(3j+5)}2}}{1-z^{3j}}
+ \sum_{j=1}^{n-1}\left(\frac{(-1)^n z^{\frac{j(3n+2)}2}}{1+ z^{\frac{3j}2}}-\frac{(-1)^j z^{\frac{j(3n+2)}2}}{1-z^{\frac{3j}2}}\right).
\end{align*}
\noindent
Combining the last two calculations, we find \eqref{mid}.
\end{proof}

\begin{lemma} If $\alpha$ is a primitive $m^{th}$-root of unity, we have
\begin{equation}\label{extan}
\sum_{k=1}^m\frac1{1-z^{-1}\alpha^k}=\frac{m}{1-z^{-m}}.
\end{equation}
\end{lemma}
\begin{proof} We introduce the function $f(z):=z^m-1=\prod_{k=1}^m(z-\alpha^k)$. Then, taking the logarithmic derivative, we obtain
$$\frac{f'(z)}{f(z)}=\sum_{k=1}^m\frac1{z-\alpha^k}$$
which means
$$\frac{m}{1-z^{-m}}=\sum_{k=1}^m\frac1{1-z^{-1}\alpha^k}.$$
\end{proof}

We set $q=\exp(\frac{2\pi ij}{3n})$ with $\gcd(j,3n)=1$.
If we apply \eqref{extan} with $\alpha=q^3$ and $z=q$ and $m=n$, there holds
\begin{align} \label{explicit} 
\sum_{k=1}^{n}\frac{1}{1-q^{3k-1}}=\frac{n}{1-q^{-n}}
=\frac{n}{3}(1-q^{n}).
\end{align}

Combining \eqref{explicit} and by using the right-hand side of \eqref{mid} with $z=q$, we put
the central declaration \eqref{main3n} in a form that is more convenient for our method of proof:
\begin{align} \label{main3n_new}
&\frac{(-1)^{n-1}}{2}\sum_{k=1}^{n-1}\frac{q^{\frac{k(3n+2)}{2}}}{1+q^{\frac{3k}{2}}} 
+\frac{1}{2}\sum_{k=1}^{n-1}\frac{(-1)^k q^{\frac{k(3n+2)}{2}}}{1- q^{\frac{3k}{2}}} + \frac{n}3(1-q^n)\nonumber\\
&\qquad -\sum_{k=1}^{\lfloor \frac{n+1}{2}\rfloor}\frac{1}{1-q^{3k-2}} -\frac{2n-1+(-1)^n}{4} = \frac13+\frac{3n+1}6\, q^{2n}.
\end{align}
Next, we proceed to study \eqref{main3n_new} by distinguishing two cases: $n=2N$ and $n=2N-1$. This allows us to circumvent fractional powers of $q$.

\medskip
\noindent 
(a) If $n=2N$ then $q^{3N}=(-1)^j=-1$ because $j$ is odd. Engaging with some algebraic simplifications,  
we find that \eqref{main3n_new} actually tantamount
\begin{align*}
&q^{2N}\sum_{k=1}^{N-1}\frac{q^{k}}{1-q^{6k}}+\sum_{k=1}^{N}\frac{q^{2k-1}}{1-q^{6k-3}}+\frac{2N}{3}(1-q^{2N})\\
&\qquad\qquad -\sum_{k=1}^{N}\frac{1}{1-q^{3k-2}} - N=\frac{1}{3}-\left(N+\frac{1}{6}\right)\,q^{N}.
\end{align*}

\noindent
(b) If $n=2N-1$ then we determine that \eqref{main3n_new} is equivalent to
\begin{align*}
&q^{2N-1}\sum_{k=1}^{N-1}\frac{q^{k}}{1-q^{6k}}+\sum_{k=1}^{N-1}\frac{q^{2k}}{1-q^{6k}}+\frac{2N-1}{3}(1-q^{2N-1})\\
&\qquad\qquad   -\sum_{k=1}^{N}\frac{1}{1-q^{3k-2}} - (N-1)= \frac{1}{3}+\left(N-\frac{1}{3}\right)\,q^{2(2N-1)}.
\end{align*}

\noindent In the next two sections, we intend to furnish the proofs for these two cases.

\section{Proof of the case $n=2N$} \label{evencase}

The condition $\gcd(j,6N)=1$ forces $j=\pm 1 \pmod{6}$. We set $\omega:=q^{N}$ so that $1-\omega+\omega^2=0$ and $\omega^3=-1$. Therefore, it suffices to show the following.

\begin{lemma} \label{firstLemma} We have that
\begin{equation}\label{even}
\omega^2 \sum_{k=1}^{N-1}\frac{q^{k}}{1-q^{6k}}+\sum_{k=1}^{N}\frac{q^{2k-1}}{1-q^{6k-3}}-\sum_{k=1}^{N}\frac{1}{1-q^{3k-2}}=-\frac{N}{3}(1+\omega)+\frac{1}{3}-\frac{\omega}{6}.
\end{equation}
\end{lemma}

\begin{proof}
We find it convenient to express our claim in terms of the quantities
\begin{align*}
&A_1=\sum_{k=1}^{N-1}\frac{1}{1-q^{k}},\quad
A_2=\sum_{k=1}^{N-1}\frac{1}{1-\omega q^{k}},\quad
A_3=\sum_{k=1}^{N-1}\frac{1}{1-\omega^2 q^{k}},\\
&A_4=\sum_{k=1}^{N-1}\frac{1}{1+q^{k}},\quad
A_5=\sum_{k=1}^{N-1}\frac{1}{1+\omega q^{k}},\quad
A_6=\sum_{k=1}^{N-1}\frac{1}{1+\omega^2 q^{k}}.
\end{align*}

\noindent 
(i) By partial fraction decomposition
\begin{align} 
\label{pfd6}
\frac{6x}{1-x^6}=\frac{1}{1-x}-\frac{\omega}{1-\omega^2 x}+\frac{\omega^2}{1+\omega x}
-\frac{1}{1+x}+\frac{\omega}{1+\omega^2 x}-\frac{\omega^2}{1-\omega x}.
\end{align}
Hence, taking $x=q^k$ results in
$$6\sum_{k=1}^{N-1}\frac{q^{k}}{1-q^{6k}}=A_1-\omega^2 A_2-\omega A_3-A_4+\omega^2 A_5+\omega A_6.$$

\noindent 
(ii) Again by partial fraction decomposition
\begin{align}
\label{pfd3}
\frac{3x}{1-x^3}=\frac{1}{1-x}-\frac{\omega}{1-\omega^2 x}+\frac{\omega^2}{1+\omega x}.
\end{align}
Thus, the choice $x=q^{2k-1}$ gives
$$3\sum_{k=1}^{N}\frac{q^{2k-1}}{1-q^{6k-3}}=B_1-\omega B_2+\omega^2B_3,$$
where
$$B_1=\sum_{k=1}^N\frac{1}{1-q^{2k-1}},\quad
B_2=\sum_{k=1}^N\frac{1}{1-\omega^2 q^{2k-1}},\quad
B_3=\sum_{k=1}^N\frac{1}{1+\omega q^{2k-1}}.$$
It is easy to check the properties $B_2=\frac{N}2$ and $B_1+B_3=N$ directly from
\begin{align*}
2\Re(B_2)=B_2+\overline{B_2}&=\sum_{k=1}^N\frac1{1+\omega^{-1}q^{2k-1}}+\sum_{k=1}^N\frac1{1+\omega q^{1-2k}},  \\
                     B_1+B_3&=\sum_{k=1}^N\frac1{1-q^{2k-1}}+\sum_{k=1}^N\frac1{1+q^Nq^{2N+2-2k-1}}.
\end{align*} 
Consequently, we obtain
$$3\sum_{k=1}^{N}\frac{q^{2k-1}}{1-q^{6k-3}}=B_1-\frac{\omega N}{2}+\omega^2(N-B_1)=(2-\omega)	\left(B_1-\frac{N}{2}\right).$$
Moreover, we recognize that
\begin{align*}
B_1&=\sum_{k=1}^{2N-1}\frac{1}{1-q^k}-\sum_{k=1}^{N-1}\frac{1}{1-q^{2k}}=A_1+\frac{1}{1-q^N}+A_2-\frac{A_1+A_4}{2} \\
&=\frac{A_1}{2}+A_2-\frac{A_4}{2}+\omega.
\end{align*}

\noindent 
(iii) Introducing the values
$$C_1:=\sum_{k=1}^N\frac{1}{1-q^{3k-1}},\quad
C_2:=\sum_{k=1}^N\frac{1}{1-q^{3k-2}},$$
we gather the property that (where we use a partial fraction of $\frac1{1-x^3}$)
\begin{align*}
C_1+C_2&=\sum_{k=1}^{3N-1}\frac{1}{1-q^k}-\sum_{k=1}^{N-1}\frac{1}{1-q^{3k}}\\
&=A_1+\frac{1}{1-q^N}+A_2+\frac{1}{1-q^{2N}}+A_3-\frac{A_1+A_3+A_5}{3}\\&=\frac{2A_1}{3}+A_2+\frac{2A_3}{3}-\frac{A_5}{3}+\frac{4\omega+1}{3},
\end{align*}
and taking advantage of \eqref{explicit} implies
\begin{align*}
\frac{2N}{3}(1-q^{2N})
&=\sum_{k=1}^{2N}\frac{1}{1-q^{3k-1}}
=\sum_{k=1}^{N}\frac{1}{1-q^{3k-1}}+\sum_{k=1}^{N}\frac{1}{1-q^{3(2N+1-k)-1}}\\
&=C_1+\sum_{k=1}^{N}\frac{1}{1-q^{-3k+2}}
=C_1 - \sum_{k=1}^{N}\frac{q^{3k-2}}{1-q^{3k-2}}\\
&=C_1+N-C_2.
\end{align*}
The last two derivation lead to
$$C_2=\frac{A_1}{3}+\frac{A_2}{2}+\frac{A_3}{3}-\frac{A_5}{6}+\frac{1+4\omega}{6}+\frac{N(2\omega-1)}{6}.$$

\noindent
Finally, by using (i), (ii), and (iii), we reduce equation  \eqref{even} to
$$\frac16\Big(-(A_1+A_6)+(A_2+A_5)-(A_3+A_4)+N-1-\omega(A_2+A_5-(N-1))\Big)=0$$
which holds true due to the symmetry $A_{\ell}+A_{7-\ell}=N-1$, for $\ell=1, 2$ and $3$.
\end{proof}

\section{Proof of the case $n=2N-1$} \label{oddcase}

\noindent Let $\omega:=-q^{2(2N-1)}=e^{\frac{\pi i}3}$ 
so that $\omega^2=q^{2N-1}$, $1-\omega+\omega^2=0$ and $\omega^3=-1$.

\begin{lemma} \label{secondLemma} We have that
\begin{equation}\label{odd}
\omega^2 \sum_{k=1}^{N-1}\frac{q^{k}}{1-q^{6k}}+\sum_{k=1}^{N-1}\frac{q^{2k}}{1-q^{6k}}-\sum_{k=1}^{N}\frac{1}{1-q^{3k-2}}=-\frac{N}{3}(1+\omega).
\end{equation}
\end{lemma}

\begin{proof}
We adopt the notations $A_i$ from the previous section.

\smallskip
\noindent (i) By the partial fraction decomposition \eqref{pfd6},
$$6q^{2N-1}\sum_{k=1}^{N-1}\frac{q^{k}}{1-q^{6k}}
=\omega^2(A_1-A_4-\omega (A_3-A_6)+\omega^2 (A_5- A_2)).$$

\noindent 
(ii) By the partial fraction decomposition \eqref{pfd3},
$$6\sum_{k=1}^{N-1}\frac{q^{2k}}{1-q^{6k}}=
A_1+A_4-\omega(A_2+A_5) +\omega^2(A_3+A_6).$$

\noindent 
(iii) We have
\begin{align*}
\sum_{k=1}^{N-1}\frac{1}{1-q^{3k-1}}+\sum_{k=1}^{N}\frac{1}{1-q^{3k-2}}
&=
\sum_{k=1}^{3N-2}\frac{1}{1-q^{k}}-\sum_{k=1}^{N-1}\frac{1}{1-q^{3k}}\\
&=A_1+\sum_{k=0}^{N-1}\frac{1}{1-q^{N+k}}+A_3-\sum_{k=1}^{N-1}\frac{1}{1-q^{3k}}\\
&=A_1+\sum_{k=0}^{N-1}\frac{1}{1-q^{2N-1-k}}+A_3-\sum_{k=1}^{N-1}\frac{1}{1-q^{3k}}\\
&=A_1+\left(N-\frac{1}{1+\omega}-A_5\right)+A_3-\frac{A_1+A_3+A_5}3\\
&=\frac{2A_1}{3}+\frac{2A_3}{3}-\frac{4A_5}{3}+N-\frac{2-\omega}{3}
\end{align*}
and invoking  \eqref{explicit} yields
\begin{align*}
\frac{2N-1}{3}(1-q^{2N-1})
&=\sum_{k=1}^{2N-1}\frac{1}{1-q^{3k-1}}=\sum_{k=1}^{N-1}\frac{1}{1-q^{3k-1}}+\sum_{k=1}^{N}\frac{1}{1-q^{3(2N-1+1-k)-1}}\\
&=\sum_{k=1}^{N-1}\frac{1}{1-q^{3k-1}}+\sum_{k=1}^{N}\frac{1}{1-q^{-3k+2}} \\
&=\sum_{k=1}^{N-1}\frac{1}{1-q^{3k-1}}-\sum_{k=1}^{N}\frac{q^{3k-2}}{1-q^{3k-2}}\\
&=\sum_{k=1}^{N-1}\frac{1}{1-q^{3k-1}}+N-\sum_{k=1}^{N}\frac{1}{1-q^{3k-2}}.
\end{align*}
The last two results imply that
%$$3\sum_{k=1}^{N}\frac{1}{1-q^{3k-2}}=N(1+\omega)+A_4-A_6.$$
\begin{align*}
6\sum_{k=1}^N\frac1{1-q^{3k-2}}
&=2A_1+2A_3-4A_5+6N+\omega-2+(2N-1)(w^2-1).
\end{align*}
Now, by (i), (ii), and (iii), we are able to restate \eqref{odd} in terms of $A_i$. So, the problem boils down to exhibiting a proof for the relation
\begin{equation}\label{odd2}
A_1+A_3-A_4-2A_5+A_6=0.
\end{equation}
Since
$$
\frac{x(1-x)}{1+x^3}=-\frac{2}{1+x}+\frac{1}{1-\omega x}+\frac{1}{1+\omega^2x}$$
we have
\begin{align*}
&\sum_{k=1}^{N-1}\frac{q^k(1-q^k)}{1+(q^k)^3}=A_6+A_2-2A_4,\\
&\sum_{k=1}^{N-1}\frac{\omega q^k(1-\omega q^k)}{1+(\omega q^k)^3}=A_1+A_3-2A_5,\\
&\sum_{k=1}^{N-1}\frac{\omega^2 q^k(1-\omega^2 q^k)}{1+(\omega^2 q^k)^3}=A_2+A_4-2A_6.
\end{align*}
Therefore we arrive at the following equivalent form of \eqref{odd2}:
\begin{equation}\label{odd3}
\sum_{k=1}^{N-1} \left(\frac{q^k(1-q^k)}{1+(q^k)^3}  +3\, \frac{\omega q^k(1-\omega q^k)}{1+(\omega q^k)^3} 
- \frac{\omega^2 q^k(1-\omega^2 q^k)}{1+(\omega^2 q^k)^3} \right)=0.
\end{equation}

\noindent
%One merit of the current formulation is that we are are able to remove $\omega$ from the denominators 
Since $1+(\omega q^k)^3=1-q^{3k}$ and $1+(\omega^2q^k)^3=1+q^{3k}$, further algebraic manipulation converts \eqref{odd3} into
\begin{align} \label{odd4} 
\sum_{k=1}^{N-1} \frac{q^k(1+\omega q^{3k})(1-\omega q^k)}{1-q^{6k}} =0.
\end{align}
On the other hand, we note that
$$\frac{z(1+\omega z^3)(1-\omega z)}{1-z^6} 
= \frac{1-\omega^2}{3}\left(\frac{1}{1-z^{-2}} + \frac{1}{1+\omega z^2} - \frac{1}{1-z^{-1}} -\frac{1}{1+\omega z}\right).$$
Letting $z=q^k$ and $\omega=-\omega^{-2}=- q^{-(2N-1)}$, our summand
can be written as
$$\frac{1-\omega^3}{3} \left(\frac{1}{1-q^{-2k}} + \frac{1}{1-
q^{-(2N-1-2k)}} - \frac{1}{1-q^{-k}} -
\frac{1}{1-q^{-(2N-1-k)}}\right).$$
Hence, the claim now becomes
$$\sum_{k=1}^{N-1}\frac{1}{1-q^{-2k}}+\sum_{k=1}^{N-1}\frac{1}{1-
q^{-(2(N-k)-1)}}=
\sum_{k=1}^{N-1}\frac{1}{1-q^{-k}}+\sum_{k=1}^{N-1}\frac{1}{1-
q^{-(2N-1-k)}}$$
and which in turn translates to
$$\sum_{k=1}^{N-1}\frac{1}{1-q^{-2k}}+\sum_{k=1}^{N-1}\frac{1}{1-q^{-(2k-1)}}
=\sum_{k=1}^{N-1}\frac{1}{1-q^{-k}}+\sum_{k=N}^{2N-2}\frac{1}{1-q^{-k}}.$$
Indeed, equality follows here since both sides of the last equation are equal to $\sum_{k=1}^{2N-2}\frac{1}{1-q^{-k}}$. In fact, this is reminiscent of the set-theoretic identity
\begin{align*}
& \  \{k: 1 \leq k \leq N-1\} \cup \{2N-1-k: 1 \leq k \leq N-1\} \\
={}& \{2k: 1 \leq k \leq N-1\} \cup \{2N-1-2k: 1 \leq k \leq N-1\}.
\end{align*}
 The proof is complete.
\end{proof}

\section{Conclusion} \label{fini}

\noindent
For the trigonometric functions enthusiast, the particular equation in \eqref{odd4} can be converted to one that involves only these circular functions. To this end, we utilize the identities
$$\frac{e^{i\theta}}{1-e^{2i\theta}}=\frac{i\csc(\theta)}{2},\qquad\frac1{1+e^{2i\theta}}=\frac{1}{2}-\frac{i\tan(\theta)}{2},$$
and we suggest  rewriting $\pi/6=\pi/2-(2N-1)x$ followed by replacing $\tan$ with $\cot$ via $\tan(\pi/2-t)=\cot(t)$. Here $x=\pi/(6N-3)$ and the outcome runs as
\begin{align*}
&\frac{q^k(1+\omega q^{3k})(1-\omega q^k)}{1-q^{6k}}
=\frac{1-\omega^2}{3}\left(\frac{q^k}{1-q^{2k}}  -\frac{1}{1+\omega q^k} + \frac{1}{1+\omega q^{2k}} \right) \\
&\qquad=  \frac{i(1-\omega^2)}6 \left(\csc(2kx)+\tan\left(\frac{\pi}6+kx\right)-\tan\left(\frac{\pi}6+2kx\right)  \right) \\
&\qquad=  \frac{i(1-\omega^2)}6 (\csc(2kx)+\cot\left((2N-1-k)x\right)-\cot\left((2N-1-2k)x\right).
\end{align*}
Hence \eqref{odd4}, reduces to verifying the trigonometric identity
$$\sum_{k=1}^{N-1} (\csc(2kx)+\cot\left((2N-1-k)x\right)-\cot\left((2N-1-2k)x\right)=0.$$

\smallskip
\noindent
For the more number-theoretic minded reader, we present below a consequence of the identity in \eqref{odd4}. We appreciate Terence Tao for allowing us to include his derivation in this paper.
For the remainder of this section, specialize to the case where $2N-1$ is coprime to $3$.

\noindent
\noindent
Introduce the cube root of unity $\epsilon: = \omega^2= e^{2\pi i/3} = q^{2N-1}$ where $q=e^{\frac{2\pi i}{3(2N-1)}}$.  Expand the numerator in equation \eqref{odd4}:
$$ \sum_{k=1}^{N-1} \frac{q^k + \epsilon^2 q^{2k} - \epsilon^2 q^{4k} - \epsilon q^{5k}}{1-q^{6k}} = 0.$$
From the easily verified ``discrete sawtooth Fourier series" identity
$$ \frac{1}{1-q^{6k}} = -\frac{1}{2N-1} \sum_{j=0}^{2N-2} j q^{6jk}$$
for any $k$ not divisible by $2N-1$ (proven by multiplying out the denominator, cancelling terms, and applying the geometric series formula), we write the preceding identity to prove as
$$ \sum_{j=0}^{2N-2} j \sum_{k=1}^{N-1} (q^{(6j+1)k} + \epsilon^2 q^{(6j+2)k} - \epsilon^2 q^{(6j+4)k} - \epsilon q^{(6j+5)k}) = 0.$$
Since $\gcd(2N-1,3)=1$, we can write $q = \epsilon^{2N-1} \zeta$ for some primitive $(2N-1)^{\mathrm{th}}$ root $\zeta$ of unity.  We then reduce to
\begin{align*}
& \sum_{j=0}^{2N-2} j\sum_{k=1}^{N-1} \left(\epsilon^{(2N-1)k} \zeta^{(6j+1)k} + \epsilon^{2(2N-1)k+2} \zeta^{(6j+2)k}\right) \\
= & \sum_{j=0}^{2N-2}j\sum_{k=1}^{N-1} \left( \epsilon^{(2N-1)k+2} \zeta^{(6j+4)k} + \epsilon^{2(2N-1)k+1} \zeta^{(6j+5)k} \right).
\end{align*}
From Galois theory we can see that the net coefficient of $\zeta^a$ would have to be independent of $a$ for each primitive residue class $a$ mod $2N-1$.  We summarize this discussion in the next declaration.

\begin{corollary}
Let $N>1$ be a natural number such that $2N-1$ is not divisible by $3$, let $\chi: \mathbb{Z} \to \mathbb{C}$ be a non-principal Dirichlet character of period $2N-1$, and let $\epsilon := e^{2\pi i/3}$.
Then, we have the character sum identity
\begin{align} \label{TaoConj}
&\left(\sum_{j=0}^{2N-2} j \cdot \chi(6j+1)\right) \left(\sum_{k=1}^{2N-2} \epsilon^{(2N-1)k} \cdot \chi(k)\right) \\
&\qquad= -\left(\sum_{j=0}^{2N-2} j \cdot \chi(6j+2)\right) \left(\sum_{k=1-N}^{N-1} \epsilon^{2(2N-1)k+2} \cdot \chi(k)\right). \nonumber
\end{align}
\end{corollary}

\noindent
We conclude with a problem proposed by Terence Tao.

\bigskip
\noindent
\bf Question. \it Is there a direct proof of the identity \eqref{TaoConj} that does not rely on \eqref{odd4}? \rm

\section*{Acknowledgment}

\noindent
The authors warmly thank ``N M'' (from Mathoverflow~\cite{MO}), Gergely Harcos and Terence Tao for valuable discussions surrounding the contents of Section~\ref{fini}.

\end{document}